\documentclass[12pt,a4paper]{amsart}
\usepackage{amsfonts}
\usepackage{amssymb}
\usepackage{amsthm}
\usepackage[centertags]{amsmath}
\usepackage{newlfont}
\usepackage{graphicx}

\newtheorem{theorem}{Theorem}
\newtheorem{proposition}[theorem]{Proposition}
\newtheorem{corollary}[theorem]{Corollary}

\def\r{\mathbb R}
\def\s{\mathbb S}

\date{}

\begin{document}
\title{Constant mean curvature surfaces with boundary on a sphere }
\author{Rafael L\'opez}
 \address{Departamento de Geometr\'{\i}a y Topolog\'{\i}a\\
Universidad de Granada\\
18071 Granada, Spain\\}
 \email{ rcamino@ugr.es}
 \author{Juncheol Pyo}
\address{Department of Mathematics\\ Pusan National University\\ Busan 609-735, Korea}
\email{jcpyo@pusan.ac.kr}

\begin{abstract}

In this article we study the shape of a compact surface of constant mean curvature of Euclidean space whose boundary is contained in a round sphere. We consider the case that the boundary is prescribed or that the surface meets the sphere with a constant angle. We study under what geometric conditions the surface must be spherical. Our results apply in many scenarios in physics where in absence of gravity a liquid drop is deposited on a round solid ball and the air-liquid interface is a critical point for area under all variations that preserve the enclosed volume.
\end{abstract}

\maketitle

\section{Introduction}

 Surfaces with constant mean curvature, abbreviated by cmc surfaces, are mathematical models of soap
films and soap bubbles and, in general, of interfaces and capillary surfaces. Under conditions of homogeneity and in absence of gravity, an interface attains a state of physical equilibrium when minimizes the surface area enclosing a fixed volume, or at least, when it is a critical point for the area under deformations that preserve the volume. The literature on capillarity is extensive and we refer the classical text of Finn \cite{fi}; applications in physics and technological processes appear in \cite{ad,gbq,lan}. We will study two physical phenomena. First, a liquid drop $W$ deposited on a surface $\Sigma$ in such way that the wetted region by $W$ on $\Sigma$ is a prescribed domain $\Omega\subset\Sigma$. Then the air-liquid interface $S=\partial W- \Omega$ is a compact cmc surface whose boundary curve $\partial S$ is prescribed to be $\partial\Omega$. The second example appears in contexts of wetting and capillarity. Consider a liquid drop $W$ deposited on a given support $\Sigma$ and such that $W$ can freely move along $\Sigma$. Here the curve $\partial S$ remains on $\Sigma$ but now is not prescribed. In equilibrium, the interface $S$ is a cmc surface and $S$ meets $\Sigma$ with a constant contact angle $\gamma$, where $\gamma$ depends on the materials.

Both contexts correspond with two mathematical problems. Denote by $S$ a compact smooth surface with boundary $\partial S$. The first problem is as follows. Given a closed smooth curve $\Gamma\subset\r^3$, study the shape of a compact cmc surface $S$ whose boundary is $\partial S=\Gamma$. For example, we ask whether the geometry of $\Gamma$ imposes restrictions to the possible configurations of $S$, such as, if the symmetries of $\Gamma$ are inherited by $S$. To be precise, suppose that $\Phi:\r^3\rightarrow\r^3$ is an isometry such that $\Phi(\Gamma)=\Gamma$ and let $S$ be a cmc surface spanning $\Gamma$. Then we ask if $\Phi(S )=S$. The simplest case of boundary is a circle contained in a plane $\Sigma$. This curve is invariant by rotations with respect to the straight line $L$ orthogonal to $\Sigma$ through the center of $\Gamma$. If $\Gamma$ is a circle of radius $r>0$ and $H\not=0$, there exist two spherical caps (the large and the small one) bounded by $\Gamma$ and with radius $1/|H|$, $0<|H|r\leq 1$ (if $|H|=1/r$, then both caps are hemispheres). Also, the planar disk bounded by $\Gamma$ is a cmc surface with $H=0$. All these examples are rotational surfaces where the axis of revolution is $L$. However, it should be noted that Kapouleas found non-rotational compact cmc surfaces bounded by a circle \cite{ka1}.

When the surface is embedded we can apply the so-called Alexandrov reflection method (or the method of moving planes) based on the maximum principle for elliptic partial differential equations of second order, which consists in a process of reflection about planes, using the very surface as a barrier \cite{al}. Thanks to this technique, given a circle $\Gamma$ contained in a plane $\Sigma$, if $S$ is a compact embedded cmc surface with $\partial S=\Gamma$ and $S$ lies on one side of $\Sigma$, then $S$ is a spherical cap. Recall that Kaopuleass examples are non-embedded surfaces and lies on one side of the boundary plane. The lack of examples provides us evidence supporting the next

\begin{quote}
Conjecture. A compact embedded cmc surface in $\r^3$ spanning a circle is a spherical cap or a planar disk.
\end{quote}

Thus, the conjecture reduces to the question under what conditions such a surface lies on one side of $\Sigma$. Some partial answers have been obtained in \cite{bemr,ko,lm}. However it is not known if there exists a cmc surface spanning a circle as in Figure \ref{fig1}, left.

 \begin{figure}[hbtp]
\begin{center}\includegraphics[width=.8\textwidth]{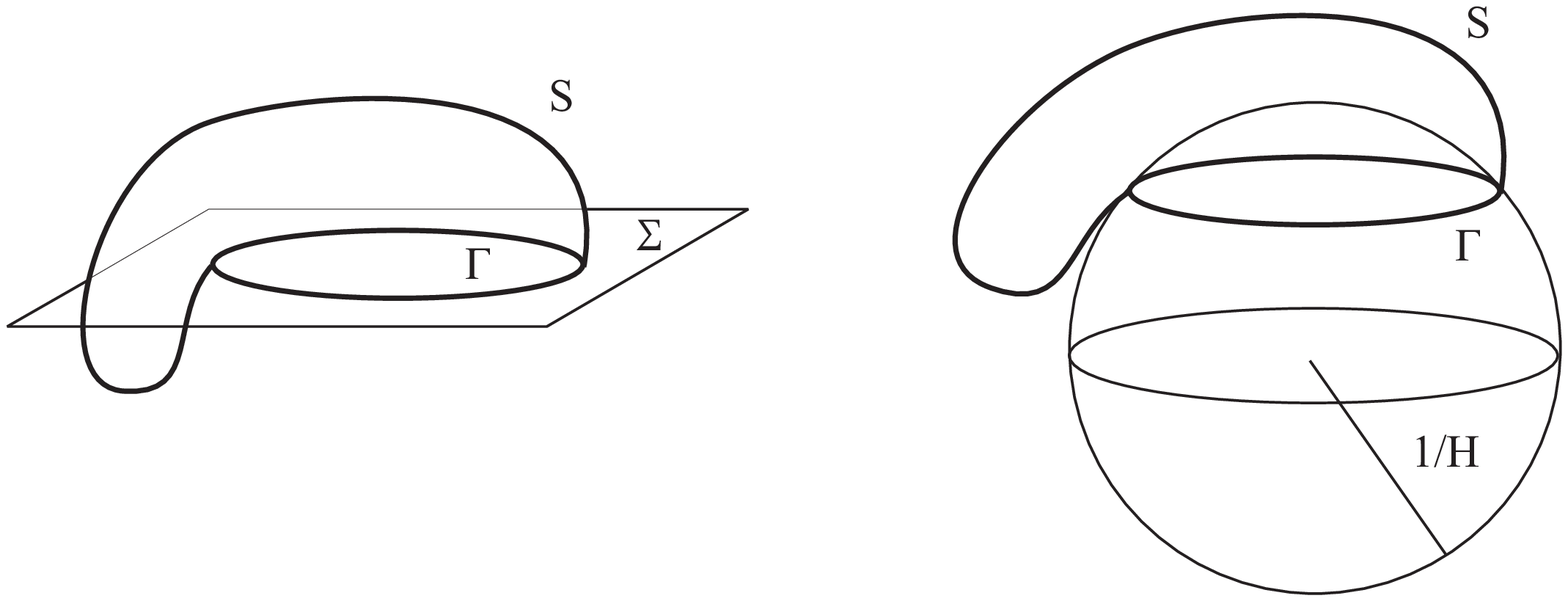}\end{center}
\caption{Left. A possible compact embedded cmc surface $S$ spanning a circle $\Gamma$. Right. This surface is not possible by Corollary \ref{cor1}}  \label{fig1}
\end{figure}

The mathematical formulation of the second setting is the following. Consider a regular region $R\subset\r^3$ with $\Sigma=\partial R$. Let $S$ be a compact surface $S$ with $int(S)\subset int(R)$ and $\partial S\subset\Sigma$ separating a bounded domain $W\subset R$ with a prescribed volume. The domain $W$ is bounded by $S$ and by pieces of $\Sigma$. Let $\gamma\in [0,\pi]$. We seek a surface $S$ which is critical for the energy functional $\mbox{area}(S)-(\cos\gamma)\mbox{ area}(\partial W\cap \Sigma)$ in the space of compact surfaces with boundary contained in $\Sigma$ and interior contained in $int(R)$ and preserving the volume of $W$. In such case, we say that $S$ is a stationary surface. A stationary surface is characterized by the fact that its mean curvature $H$ is constant and $S$ meets $\Sigma$ in a constant angle $\gamma$ along $\partial S$.

 In this article we shall consider both problems when the supporting surface $\Sigma$ is a sphere. First in Section \ref{sec2} we study compact embedded cmc surfaces with prescribed boundary on a sphere $\Sigma$. We give results showing that the surface inherits some symmetries of its boundary, and we prove that the Conjecture is true in some special cases. See Figure \ref{fig1}, right. In Section \ref{sec3}, we study stationary surfaces with boundary on a sphere. In physics, these configurations appear in the context of capillarity, for example, \cite{ed,kiit,mc,osr,vs,wmjs}. In the case that we study here, say, $\Sigma$ is a sphere, a result of Taylor asserts that the boundary of a stationary surface is smooth because of $\Sigma$ \cite{ta}. With the above notation, if $R$ is the closed ball defined by $\Sigma$, there are examples of stationary surfaces intersecting $\Sigma$ with a  contact angle: besides the planar disks and spherical caps, whose boundary is a circle contained in $\Sigma$, there are pieces of rotational (non spherical) cmc surfaces whose boundary is formed by two coaxial circles \cite{de}. Nitsche proved that the only cmc surface homeomorphic to a disk that meets $\Sigma$ at a contact angle is either a planar disk or a spherical cap (\cite{ni}; also \cite{rs}). By the physical interest, we also study the case that the mean curvature depends linearly on a spatial coordinate.

\section{Surfaces with prescribed boundary in a sphere}\label{sec2}

Let $\Gamma$ be a boundary curve, possibly disconnected, on a sphere of radius $\rho$ and centered at the origin, which will be denoted by $\s_\rho$. We want to study the shape of a compact embedded cmc surface $S$ such that its boundary curve $\partial S$ is $\Gamma$. The simplest example is when $\Gamma$ is a circle $\Gamma\subset\s_\rho$. Then there is a planar disk ($H=0$) spanning $\Gamma$ and a family of spherical caps bounded by $\Gamma$. A second example appears when $\Gamma$ lies in an open hemisphere of $\s_\rho$. Then, under conditions on mean convexity of $\Gamma$, it is possible to construct radial graphs on a domain of the hemisphere and spanning $\Gamma$ \cite{lo2,se}. By a radial graph on a domain $\Omega\subset \s_\rho$, also called a surface with a one-to-one central projection on $\Omega$, we mean a surface $S$ such that any ray emanating from the origin and crossing $\Omega$ intersects $S$ once exactly.

We shall use the Hopf maximum principle of elliptic equations of divergence type, that in our context of cmc surfaces, we call the tangency principle.

\begin{proposition}[Tangency principle] Let $S_1$ and $S_2$ be two surfaces of $\r^3$ that are tangent at some common point $p$. Assume that $p\in int(S_1)\cap int(S_2)$ or $p\in\partial S_1\cap\partial S_2$. In the latter case, we further assume that $\partial S_1$ and $\partial S_2$ are tangent at $p$ and both are local graphs over a common neighborhood in the tangent plane $T_p S_1=T_pS_2$. Consider on $S_1$ and $S_2$ the unit normal vectors agreeing at $p$. Assume that with respect to the reference system determined by the unit normal vector at $p$, $S_1$ lies above $S_2$ around $p$, which be denoted by $S_1\geq S_2$. If $H_1\leq H_2$ at $p$, then $S_1$ and $S_2$ coincide in an open set around $p$.
\end{proposition}

A result by the first author on embedded cmc surfaces with boundary in a sphere appears in \cite{lo}, which extends a previous result of Koiso \cite{ko}. Let $B_\rho$ be the open ball bounded by $\s_\rho$ and $E_\rho=\r^3-\overline{B_\rho}$.

\begin{theorem}[\cite{lo}] Let $H\not=0$ and let $S$ be a compact embedded surface in $\r^3$ with constant mean curvature $H$. Assume that $\Gamma=\partial S$ is a simple closed curve included in an open hemisphere of the sphere $\s_{1/|H|}$ and denote by $\Omega\subset \s_{1/|H|}$ the domain bounded by $\Gamma$ in this hemisphere. If $S\cap (\s_{1/|H|}-\overline{ \Omega})=\emptyset$, then either
$S-\Gamma=\Omega$, or $S-\Gamma\subset B_{1/|H|}$ or
$S-\Gamma\subset E_{1/|H|}$.
\end{theorem}

In this context, we shall prove a general criterion that, under some conditions, the surface inherits the symmetries of its boundary. The next theorem extends to cmc surfaces whose boundary lies on  a sphere, previous results of symmetry by Koiso and Schoen for cmc surfaces with planar boundary \cite{ko,sc}.

\begin{theorem}\label{t1} Let $\rho\not=0$. Let $\Gamma$ be a closed curve, possibly disconnected, included in an open hemisphere of $\s_{\rho}$. Assume that $P$ is a vector plane through the pole of the hemisphere which it is plane of symmetry of $\Gamma$ and $P$ separates $\Gamma$ in two graphs on $P\cap \s_{\rho}$. Suppose that $S$ is a compact embedded surface spanning $\Gamma$ with constant mean curvature $H$ and $|H|\geq 1/\rho$. If $S-\Gamma$ is included in $E_{\rho}$, then $S$ is symmetric with respect to $P$.
\end{theorem}

Here we say that $P$ separates $\Gamma$ in two graphs if when we write $\Gamma-(\Gamma\cap P)=\Gamma^+\cup\Gamma^{-}$, being $\Gamma^+$ and $\Gamma^{-}$ the parts of $\Gamma$ on each side of $P$, any circle of $\s_{\rho}$ contained in an orthogonal plane to $P$ and parallel to the axis through the pole of the hemisphere intersects $\Gamma^+$ and $\Gamma^{-}$ once at most.

\begin{proof} After a rigid motion, assume that the plane $P$ is of equation $x=0$, where $(x,y,z)$ are the usual coordinates of $\r^3$, and that the hemisphere of $\s_{\rho}$ is $\s_{\rho}^{+}=\s_{\rho}\cap\{(x,y,z)\in\r^3: z>0\}$.
Let $\Gamma=\Gamma_1\cup\ldots\cup \Gamma_n$ be the decomposition of $\Gamma$ into its connected components and denote by $\Omega_i\subset \s_{\rho}^+$ the domain bounded by $\Gamma_i$ on $\s_{\rho}^+$. Since $\Gamma^+$ and $\Gamma^-$ are graphs on $P\cap\s_{\rho}$ then $\Omega_i\cap\Omega_j=\emptyset$, $i\not=j$. Set
$M=S\cup(\overline{\Omega_1}\cup\ldots\cup\overline{\Omega_n})$, which is a connected compact topological surface without boundary of $\r^3$. Therefore, $M$ bounds a $3$-domain $W\subset\r^3$ by the Alexander duality theorem \cite{gra}. Consider on $S$ the unit normal vector field $N$ pointing towards $W$.

We claim that with this choice of $N$, the mean curvature $H$ of $S$ is positive. Indeed, let $p_0\in S$ be the fairest point  from the origin and let $T$ be the affine tangent plane of $S$ at $p_0$.  Choose on $T$ the unit normal vector that coincides with $N$ at $p_0$. Since $N(p_0)$ points towards $W$, then $N(p_0)=-p_0/|p_0|$. With respect to the reference system given by $N(p_0)$,  $S\geq T$ around $p_0$. The tangency principle gives $H>0$ because $T$ is a minimal surface. This proves the claim.

Since $M\cap B_{\rho}=\emptyset$, we have two possibilities about the domain $W$, namely, $B_{\rho}\subset\r^3-\overline{W}$ or $B_{\rho}\subset W$.

\begin{figure}[hbtp]
\begin{center} \includegraphics[width=.9\textwidth]{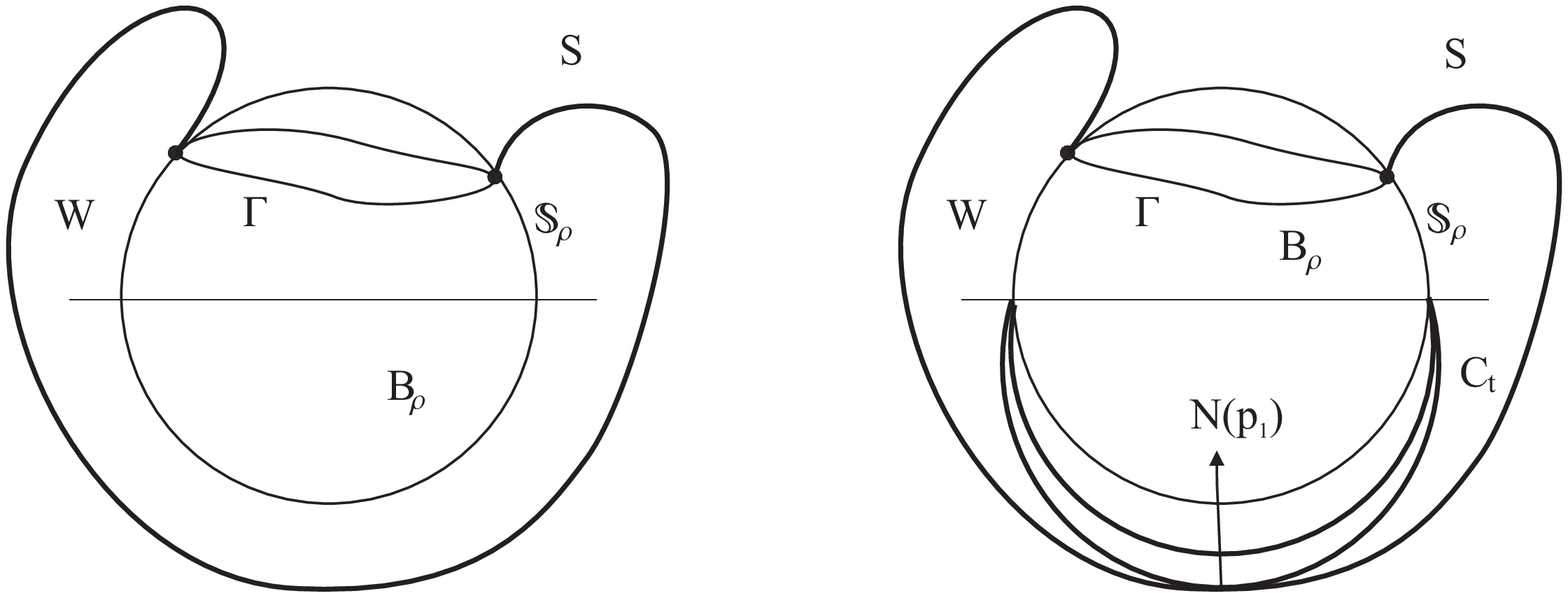}\end{center}
\caption{Left. The case that $B_\rho$ is included in the domain $W$. Right. A contradiction is obtained by comparing $S$ with spherical caps $C_t$.}  \label{fig2}
\end{figure}

We prove that the case $B_\rho\subset W$ is not possible. See Figure \ref{fig2}, left. If this occurs, take the uniparametric family of spherical caps $C_t$ contained in $\overline{E_{\rho}}\cap \{(x,y,z)\in\r^3: z\leq 0\}$, $t\in [\rho,\infty)$, whose boundary is the equator $\s_{\rho}\cap \{(x,y,z)\in\r^3: z=0\}$ and where the parameter $t$ indicates the radius of $C_t$. Let us fix the unit normal vector on $C_t$ pointing towards the center of $C_t$. For $t=\rho$, $C_\rho$ is the lower hemisphere $\s_{\rho}-\s_{\rho}^{+}$. Starting from $t=\rho$, increase the radius of $C_t$. For values $t$ close to $t=\rho$, it holds $int(C_t)\subset W$ by the property that $\Gamma\subset
\s_{\rho}^{+}$. Since $W$ is bounded, we increase $t$ until $C_{t}$ first intersects $S$ at $t=t_1>\rho$, as depicted in Figure \ref{fig2}, right. This must occur at some interior point $p_1$ of both surfaces. The unit normal vectors of $C_{t_1}$ and $S$ agree at $p_1$ because both vectors point towards $W$. Now,  $C_{t_1}\geq S$ around $p_1$ but the mean curvature of $C_{t_1}$ is $1/t_1$ and $1/t_1<1/\rho\leq H$, in contradiction with the tangency principle.

As a consequence, the ball $B_{\rho}$ is included in $\r^3-\overline{W}$, or equivalently, $W\subset E_{\rho}$. Apply  now the Alexandrov reflection method with two families of planes. First, let $\Pi(t)$ be the plane of equation $z=t$ and let $\{\Pi(t): t\in \r\}$ be the family of horizontal planes. Since $W$ is a bounded set, let $t<0$ be so $W\subset \{(x,y,z)\in\r^3: z>t\}$. Next, increase $t$ until $\Pi(t)$ first contacts with $S$ at $t=t_1$, i.e. $\Pi(t)\cap S=\emptyset$ for $t<t_1$ and $\Pi(t_1)\cap S\not=\emptyset$. Then $t_1\leq \bar{t}:=\min\{z(q): q\in\Gamma\}$ and $\bar{t}>0$, where $z(q)$ denotes the $z$-coordinate of $q$. Next, increase $t$ and reflect the part of $S$ below $\Pi(t)$ with respect to $\Pi(t)$. In order to be precise, let us introduce the next notation. Denote by $\Phi_t$ the reflection about $\Pi(t)$. Let
$$S(t)^{-}=\{(x,y,z)\in S: z\leq t\}, \ S(t)^+=\{(x,y,z)\in S: z\geq t\}$$
and $S(t)^*=\Phi_t(S(t)^{-})$.

There are two possibilities about $t_1$: $t_1<\bar{t}$ or $t_1=\bar{t}$.

Assume first that $t_1<\bar{t}$ and so $\Pi(t_1)$ intersects $S$ at some interior point of $S$. As $\Pi(t_1)$ and $S$ are tangent at the touching points, there exists $\epsilon>0$ such that for all $t\in (t_1,t_1+\epsilon)$, we have $int(S(t)^{*})\subset W$ and $S(t)^{*}$ is a graph on some domain of $\Pi(t)$. We keep increasing $t$ and reflecting $S(t)^-$ until the first time $t_2$ such that either $int(S(t_2)^*)\not\subset W$ or $S(t_2)^*$ is not a graph on $\Pi(t_2)$. The existence of $t_2$ is guaranteed since $W$ is bounded. First, assume that $t_2\leq 0$. There are two possibilities:
\begin{enumerate}
\item There exists a point $p_2\in int((S(t_2)^*)\cap S(t_2)^{+}$. Then $\Phi_{t_2}(q_2)=p_2$ for some $q_2\in S(t_2)^{-}$, with $q_2\not=p_2$. In particular, and by the reflection process, the vertical open segment joining $q_2$ with $p_2$ lies in $W$. Then $p_2\not\in \Gamma$ because $B_{\rho}\subset\r^3-\overline{W}$. This implies that $p_2$ is an interior point of $S(t_2)^+$. Thus $S(t_2)^*$ and $S(t_2)^{+}$ are tangent at $p_2$ and $S(t_2)^*\geq S(t_2)^+$ around $p_2$ since $N$ points towards $W$. Then the tangency principle assures that $S(t_2)^+=S(t_2)^*$ in an open set of $p_2$.

\item There exists $p_2\in \partial S(t_2)^*\cap \partial S(t_2)^{+}$. As $t_2\leq 0<\bar{t}$, $p_2\not\in \Gamma$. Then $S(t_2)^*$ and $S(t_2)^{+}$ are tangent at $p_2$, where $p_2\in \Pi(t_2)$, the tangent plane $T_p S$ is orthogonal to $\Pi(t_2)$ and the boundaries $\partial S(t_2)^+$, $\partial S(t_2)^*$ are tangent at $p_2$. Since $S(t_2)^+$ and $S(t_2)^*$ are graphs in some domain of $T_pS \cap \{(x,y,z)\in\r^3: z\geq t_2\}$, the boundary version of the tangency principle implies that $S(t_2)^{+}=S(t_2)^*$ in an open set of $p_2$.
\end{enumerate}
In both cases, denote by $K^{*}$ the connected component of $S(t_2)^*$ which contains $p_2$. Then $K^{*}$ is the reflection about $\Pi(t_2)$ of a certain connected component $K$ of $S(t_2)^{-}$, that is, $K^{*}=\Phi_{t_2}(K)$. As $S(t_2)^*\cap \Gamma=\emptyset$, we repeat the argument with the tangency principle and we derive that $K^{*}\cap S$ is an open and closed subset of $K^{*}$. Since $K^{*}$ is connected, $K^{*}\cap S=K^{*}$. Hence $K^{*}$ is contained in $S$. Since $S(t_2)^*\cap\Gamma=\emptyset$, it follows that $K^{*}\subset S-\Gamma$. Therefore $\partial K$ coincides with $\partial K^{*}$. This proves that $\overline{K\cup K^{*}}$ is a compact surface without boundary contained in $S$ and this implies that it coincides with $S$, which contradicts the assumption that $\partial S=\Gamma$.

The above reasoning implies that the reflection process with horizontal planes $\Pi(t)$ arrives until the plane $\Pi(0)$ and, furthermore, $S(0)^*$ lies in $W$.

Once arrived here, and in the reflection process, we change the family of planes with respect to we make the reflections.  Denote by $Q(t)$ the plane orthogonal to $(\cos(t),0,\sin(t))$ for $t\in [\pi/2,\pi]$ whose equation is $\cos(t)x+\sin(t)z=0$. Let $\{Q(t): t\in[\pi/2,\pi]\}$. All planes $Q(t)$ have a common straight line $L$, namely, of equations $x=z=0$. The plane $Q(\pi/2)$ coincides with $\Pi(0)$ and $Q(\pi)$ coincides with $P$. See Figure \ref{fig3}. Let $\Psi_t$ be the reflection about $Q(t)$. Similarly as with the planes $\Pi(t)$, let us introduce the next notation.
$$S_Q(t)^{-}=\{(x,y,z)\in S: \cos(t)x+\sin(t)z\leq 0\}$$
$$S_Q(t)^+=\{(x,y,z)\in S: \cos(t)x+\sin(t)z\geq  0\}$$
and $S_Q(t)^*=\Psi_t(S_Q(t)^-)$, the reflection of $S_Q(t)^{-}$ about $Q(t)$.  Also, let $\Gamma^-=\Gamma\cap \{(x,y,z)\in\r^3: x> 0\}$ and $\Gamma^+=\Gamma\cap \{(x,y,z)\in\r^3: x<0\}$.

The idea now is to follow the Alexandrov reflection method using the planes $Q(t)$ starting at $t=\pi/2$ and showing that we reach the value $t=\pi$ in the reflection process. First, remark some properties of the reflections $\Psi_t$.

\begin{enumerate}
\item Since $\Psi_{\pi}$ is the reflection about the plane $P$, we have $\Psi_{\pi}(\Gamma^{-})=\Gamma^{+}$.
\item It is possible to replace the family $\Pi(t)$ by $Q(t)$. Indeed, for $t=\pi/2$, the plane $Q(\pi/2)$ is $\Pi(0)$, the surface $S(0)^*$ coincides with $S_Q(0)^*$ and thus, $S_Q(0)^*$ is included in $W$.
\item The reflections about $Q(t)$ leave $\s_{\rho}$ invariant.
\item For each $t\in [\pi/2,\pi)$, if $Q(t)$ contains a point $p\in\Gamma$, then $p\in\Gamma^{-}$. Moreover, the set $\{\Psi_t(p): t\in [\pi/2,\pi)\}$ describes a piece of a half circle in $\s_{\rho}^{+}$ contained in a plane orthogonal to $P$, parallel to the $z$-axis and intersecting $\Gamma$ exactly at one point, namely, $p\in\Gamma^-$. This is due to the fact that $\Gamma^-$ and $\Gamma^{+}$ are graphs on $P\cap\s_{\rho}$.
\end{enumerate}

Let us increase $t\nearrow \pi$ until the first time $t_3$ for which $S_Q(t_3)^*$ leaves to be included in $W$ or   leaves to be a graph on $Q(t_3)$.

We claim that $t_3=\pi$. The argument is by contradiction and let us assume $t_3<\pi$. Then $S_Q(t_3)^*$ intersects $S_Q(t_3)^{+}$ at a first point $p_3$. This point $p_3$ cannot belong to $\Gamma^{+}$ because in a such case,  $p_3=\Psi_{t_3}(q_3)$ for some $q_3\in \s_\rho\cap S_Q(t_3)^{-}$ and then necessarily $q_3\in\Gamma^{-}$. This is a contradiction because $t_3<\pi$ and $Q(t_3)$ is not a plane of symmetry of $\Gamma$. This proves that $p_3$ is a common interior point of $S_Q(t_3)^+$ and $S_Q(t_3)^{*}$ or a boundary point of both surfaces. In particular, $S_Q(t_3)^{*}$ and $S_Q(t_3)^{+}$ are tangent at $p_3$ and $S_Q(t_3)^{*}\geq S_Q(t_3)^{+}$ around  $p_3$. An argument using the tangency principle, similar as in the part of reflections with the planes $\Pi(t)$, implies that $Q(t_3)$ is a plane of symmetry of $S$, in particular, of $\Gamma$: a contradiction because $t_3<\pi$.

\begin{figure}[hbtp]
\begin{center} \includegraphics[width=.6\textwidth]{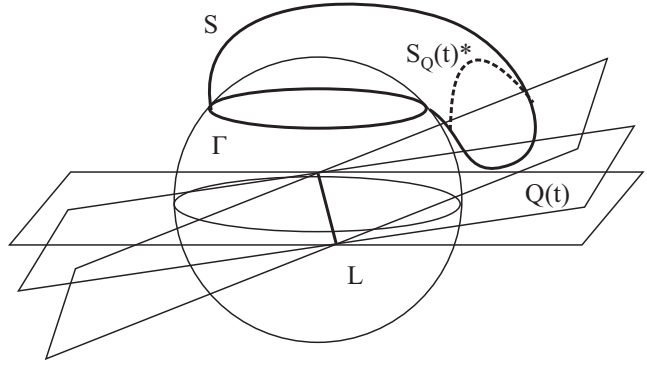}\end{center}
\caption{Proof of Theorem \ref{t1}. Reflections about the planes $Q(t)$. }\label{fig3}
\end{figure}

After the claim, we have proved $t_3=\pi$. There are two possibilities. First possibility is that
$S_Q(\pi)^*-\Gamma^{-}$ is completely included in $W$. In such case, we can do an analogous process of reflection about the planes $Q(t)$ with $t\in [0,\pi/2]$ starting from $t=\pi/2$ and letting $t\searrow 0$. The same argument and the fact that $S_Q(\pi)^*-\Gamma^{-}
\subset W$ would imply the existence of $t_4\in (0,\pi/2)$ such that $Q(t_4)$ is a plane of symmetry of $S$, which is not possible. Therefore, the only possibility is that $S_Q(\pi)^*-\Gamma^{-}$ touches $\partial W$ at some point, which necessarily belongs to $S_Q(\pi)^{-}$. The tangency principle implies that $Q(\pi)=P$, proving the theorem.

The reasoning carried out until here assumed that $t_1<\bar{t}$. The other possibility is that $t_1=\bar{t}$ or in other words, by displacing $\Pi(t)$ upwards vertically, we do not intersect $S$ until $\Pi(t)$ arrives to the value $t=\bar{t}$, the lowest height of $\Gamma$. Since $\bar{t}>0$, we move back $\Pi(t)$ until the position $t=0$ and we begin the reflection method using the planes $Q(t)$ as above. The difference with respect to the above process initiated with the planes $Q(t)$ is that in the starting point, we now have $Q(\pi/2)\cap S=\emptyset$. Anyway, let us increase $t$ until we arrive the first point $t_1\in (\pi/2,\pi)$ of contact between $Q(t_1)$ and $S$. Next, continue with $t$ and reflecting about $Q(t)$ until the first time $t_2$ that $S_Q(t_2)^*$ leaves the domain $W$ or $S_Q(t_2)^{*}$ leaves to be a graph on $S_Q(t_2)$. Now the proof follows the same steps, proving the result. $\Box$
\end{proof}

In the case that $\Gamma$ is a circle, we conclude

\begin{corollary} \label{cor1} Let $H\not=0$ and let $\Gamma$ be a circle with radius $r>0$, $r\leq 1/|H|$. Suppose that $S$ is a compact embedded surface spanning $\Gamma$ and with constant mean curvature $H$ such that $S-\Gamma$ is included in $E_{1/|H|}$, where $\s_{1/|H|}$ is any of the two spheres of radius $1/|H|$ containing $\Gamma$. Then $S$ is a spherical cap.
\end{corollary}

\begin{proof} If $r=1/|H|$, $S$ is a hemisphere of radius $1/|H|$ by a result of Brito and Earp \cite{br}. But this hemisphere is not included in $E_{1/|H|}$ and this case is not possible. Therefore, $r<1/|H|$. Let $\Pi$ be the plane through the center of $\s_{1/|H|}$ and parallel to the one containing $\Gamma$. Then $\Pi$ defines a unique open hemisphere containing $\Gamma$. Let $L$ be the straight line orthogonal to $\Pi$ through the center of $\Gamma$. Apply Theorem \ref{t1} for any plane containing $L$, which is a plane of symmetry of $\Gamma$ and separates $\Gamma$ in two graphs. Then we conclude that $S$ is invariant by the group of rotations with axis $L$. This shows that $S$ is a surface of revolution. By the classification of Delaunay \cite{de}, the only compact rotational non-zero cmc surface spanning a circle is a spherical cap. $\Box$
\end{proof}

As an example, Corollary \ref{cor1} says that the surface that appears in Figure \ref{fig1}, right, is not possible, where $H$ is the (constant) mean curvature of $S$.

In the proof of Theorem \ref{t1}, notice that the assumption $S-\Gamma\subset E_{\rho}$ has been used before to the Alexandrov method in order to prove that $B_{\rho}$ is not included in $W$. Therefore, in the next result we drop the assumption on the value of the radius of the sphere.

\begin{corollary}\label{cor2} Let $\Gamma$ be a closed curve, possibly disconnected, included in an open hemisphere of a sphere $\s_\rho$ and assume that $P$ is a plane of symmetry of $\Gamma$ which separates $\Gamma$ in two graphs defined on $P\cap \s_\rho$. Suppose that $S$ is a compact embedded cmc surface spanning $\Gamma$ and $S$ satisfies one of the following two hypothesis:
\begin{enumerate}
\item The surface $S-\Gamma$ is included in $B_\rho$.
\item The bounded domain $W$ defined by $S\cup(\overline{\Omega}_1\cup\ldots\cup\overline{\Omega}_n)$ is included in $E_\rho$.
\end{enumerate}
 Then $S$ is symmetric with respect to the plane $P$. In the case that $\Gamma$ is a circle, the surface is a spherical cap.
\end{corollary}

\begin{proof} By an isometry in $\r^3$, we may assume that $P=\{(x,y,z)\in\r^3: x=0\}$ and the open hemisphere is $\s^{+}_{\rho}$. The Alexandrov method works as in Theorem \ref{t1} beginning with the horizontal planes $\Pi(t)$. We only sketch the proof remarking the differences and only for the case $S-\Gamma\subset B_\rho$. We can arrive at least until $t=0$ and no contact points appear between $S(t)^*$ and $S(t)^+$: on the contrary, the tangency principle would give a plane of symmetry $\Pi(t_2)$ for some $t_2\leq 0$, which is not possible. In this part of the proof, we point out that it is not possible that a point of $S(t)^*$ touches a point of $\Gamma$ because the reflection of points of $S\cap B_\rho$ lies in $B_\rho$. Once arrived at $t=0$, use reflections about the planes $Q(t)$, obtaining the result. $\Box$
\end{proof}

We get another consequence of the proof of Theorem \ref{t1}. For this aim, we need a result that will be also useful in the rest of this work.

\begin{proposition}\label{pr1} Let $S$ be a compact embedded surface whose boundary lies in an open hemisphere of a sphere $\s_\rho$. Suppose that $S-\partial S$ is included in $E_\rho$ or in $B_\rho$. Then there exists a subset $\Omega\subset\s^+$ such that $S\cup \Omega$ is a closed embedded surface of $\r^3$. In particular, $S\cup\Omega$ defines a bounded $3$-domain $W\subset\r^3$.
\end{proposition}

In fact the result is more general and it holds for compact embedded orientable surfaces such that $\partial S\subset \Sigma$, where $\Sigma$ is the boundary of a simply connected domain $D$ of $\r^3$ with one end and $S-\partial S\subset D$. The difference in Proposition \ref{pr1} is that there are two domains, one is determined by $S\cup\Omega$ and the other one by $S\cup(\s_\rho- \Omega)$.

\begin{corollary} Let $S$ be a compact embedded surface of constant mean curvature $H$, $H\not=0$. Assume that $\partial S$ is included in an open hemisphere of a sphere $\s_{\rho}$ with $|H|\geq 1/\rho$. If $S-\Gamma\subset E_{\rho}$, then the domain $W$ defined in Proposition \ref{pr1} is included in $E_{\rho}$.
\end{corollary}

\begin{proof} If $W\not\subset E_{\rho}$, then $B_{\rho}\subset W$. The first part of the proof of Theorem \ref{t1} holds proving that, according to the unit normal vector of $S$ pointing towards $W$, the mean curvature $H$ is positive. Using the spherical caps $C_t$ that appeared there, a similar argument of comparison yields a contradiction. $\Box$
\end{proof}

\section{Capillary surfaces with boundary in a sphere}\label{sec3}

By a \emph{capillary surface} on a support surface $\Sigma$ we mean a compact embedded surface $S$ of constant mean curvature whose boundary lies on $\Sigma$ and the angle $\gamma$ between $S$ and $\Sigma$ along $\partial S$ is constant. If the boundary curve is included in a plane $\Sigma$ and $S$ lies on one side of $\Sigma$, then $S$ is a spherical cap. In fact, Wente proved a more general result. Consider  the usual coordinates $(x,y,z)$ of $\r^3$ and denote by $z(p)$ the third coordinate of a point $p\in\r^3$.

\begin{theorem}[\cite{we0}]\label{thw} Let $\kappa,\mu\in\r$. Assume that $S$ is a compact embedded surface with mean curvature $H(x,y,z)=\kappa z+\mu$ whose boundary lies in a horizontal plane $\Sigma$ and $S$ meets $\Sigma$ with constant angle. If $S$ lies on one side of $\Sigma$, then $S$ is a surface of revolution. Moreover, any nonempty intersection of $S$ with a horizontal plane is a circle with center on the axis.
\end{theorem}

Surfaces in Euclidean space whose mean curvature $H$ satisfies $H(x,y,z)=\kappa z+\mu$ are mathematical models of interfaces under the presence of gravity. The proof of Theorem \ref{thw} uses the Alexandrov method by vertical planes and the fact that the reflection about these planes do not change the value of the mean curvature because the third coordinate remains invariant under these reflections. A difference with respect to the case of prescribed boundary is that a new case of contact point $p$ appears in the reflection process. This occurs when this point is a common boundary point between the reflected surface and the original one and, moreover, $p\in\partial S$. The hypothesis on the contact angle implies that both surfaces are tangent at $p$ and both surfaces may be expressed locally as a graph over a quadrant in the common tangent plane at $p$ and where $p$ is the corner of the quadrant. Then we apply a version of the Hopf maximum principle called the Serrin corner lemma. The tangency principle holds for this case, showing  that both surfaces agree in an open set around $p$ \cite{se2}.

Consider $\Sigma$ a sphere $\s_\rho$. Examples of capillary surfaces on $\s_\rho$ are planar disks, spherical caps and rotational surfaces spanning two coaxial circles of $\s_\rho$. Exactly, some pieces of catenoids, unduloids and nodoids contained in $B_\rho$ meet $\s_\rho$ in two circles making a constant contact angle. In this section we study compact embedded surfaces whose boundary lies in an open hemisphere of a sphere $\s_{\rho}$. In order to fix the notation, denote by $W\subset\r^3$ the bounded $3$-domain by $S\cup\Omega$ given by Proposition \ref{pr1}. The following result extends Theorem \ref{thw} for capillary surfaces whose boundary lies in an open hemisphere and it uses the Alexandrov reflection as in Theorem \ref{t1}.

\begin{theorem}\label{t4} Let $S$ be a capillary surface on a sphere $\s_\rho$ such that its boundary $\Gamma$ lies in an open hemisphere. Assume one of the next two hypothesis:
\begin{enumerate}
\item $S-\Gamma$ is included in the ball $B_\rho$.
\item $S-\Gamma$ and $W$ are included in $E_\rho$.
\end{enumerate}
Then $S$ is a spherical cap.
\end{theorem}

\begin{proof} The first case was showed in \cite[Prop. 1.2]{rs}. Suppose the second case. After a rigid motion, assume that the open hemisphere containing $\Gamma$ is $\s_{\rho}^+$. Let $\Pi$ be the plane of equation $z=0$ and consider on $S$ the unit normal vector $N$ that points towards $W$. Let us take a fixed horizontal straight line $L$ through the origin. After a rotation about the $z$-axis, suppose that $L$ is the $y$-axis. The Alexandrov method works as in Theorem \ref{t1} beginning by reflections about the planes of type $\Pi(t)$ and the argument is similar. Take the notation used there. In this part of the proof, we prove that there are no contact tangent points between the reflected surface and the initial surface.

Next, we follow using the planes $Q(t)$ where now $t\in [\pi/2,3\pi/2]$. Observe that all the planes $Q(t)$ contain $L$. Let us use the same notation as in Theorem \ref{t1}. Start with $t=\pi/2$ and let $t\nearrow 3\pi/2$. Remark that $Q(\pi/2)=Q(3\pi/2)$ and as the parameter $t$ runs in the interval $[\pi/2,3\pi/2]$, the planes $Q(t)$ sweep all the hemisphere $\s_{\rho}^{+}$. Therefore there exists a first contact point at $t_3\in (\pi/2,3\pi/2)$. We have the next types of contact points:
\begin{enumerate}
\item The surface $S_Q(t_3)^*$ intersects $S_Q(t_3)^+$ at some common interior or boundary point $p_3$ that does not belong to $\Gamma$ and both surfaces are tangent at $p_3$. Then the tangency principle asserts that $Q(t_3)$ is a plane of symmetry of the surface.
\item The contact point $p_3$ lies in $\partial S_Q(t_3)^*\cap \partial S_Q(t_3)^+\cap\Gamma$. As in Theorem \ref{thw}, the condition that $S$ meets $\s_\rho$ with constant angle implies that both surfaces $S_Q(t_3)^*$ and $S_Q(t_3)^+$ are tangent at $p_3$. Furthermore, $S_Q(t_3)^*$ and $S_Q(t_3)^{+}$ are graphs on a quadrant of $T_{p_3}S$ where $p_3$ is the corner of the domain. Use the Serrin corner lemma to conclude that $S_Q(t_3)$ is a plane of symmetry.
\end{enumerate}
Summarizing, we have proved that for each horizontal straight line $L$ through the origin of coordinates, there exists a vector plane $Q_L$ containing $L$ such that $S$ is symmetric about $Q_L$. Since the surface is compact, necessarily all these planes $Q_L$ have a common straight line $R$. Therefore, $S$ is a (non planar) surface of revolution about the axis $R$ and since the mean curvature is constant, $S$ is a spherical cap. $\Box$
\end{proof}

Theorem \ref{t4} extends when the mean curvature is a linear function on a coordinate of $\r^3$.

\begin{theorem}\label{t5} Let $\kappa,\mu\in\r$ with $\kappa>0$. Let $S$ be a compact embedded surface such that its boundary $\Gamma$ lies in $\s_{\rho}^+=\{(x,y,z)\in\s_\rho: z>0\}$. Assume
 \begin{enumerate}
\item $S-\Gamma$ is included in the ball $B_\rho$ or $S-\Gamma$ and $W$ are included in $E_\rho$.
\item The surface $S$ meets $\s_\rho$ with a constant angle along $\Gamma$.
\item The mean curvature of $S$ with respect to the unit normal vector pointing $W$ satisfies $H(x,y,z)=\kappa z+\mu$.
\end{enumerate}
Then $S$ is a surface of revolution with respect to the $z$-axis and any nonempty intersection of $S$ with a horizontal plane is a circle with center on the $z$-axis.
\end{theorem}

\begin{proof} The proof for the cases that $S-\Gamma\subset B_\rho$ or $S-\Gamma, W\subset E_\rho$ are similar and we only consider the case that $S-\Gamma$ is included in $B_\rho$. In particular, $W\subset B_\rho$. The proof follows the same steps as in Theorem \ref{t4} and we use the same notation. Fix $L\subset\Pi$ a horizontal straight line through the origin which, after a rotation about the $z$-axis, suppose that $L$ is the $y$-axis. Start with reflections about the planes $\Pi(t)$. Assume $t_2<0$. The contact point $p_2$ is a common interior or boundary point of $S(t_2)^*$ and $S(t_2)^+$ where both surfaces are tangent. Then $p_2=\Phi_{t_2}(q_2)$. In particular, $z(q_2)\leq z(p_2)$. Denote by $H^*$ the mean curvature of the reflected surface $S(t_2)^*$. The unit normal vectors of $S(t_2)^*$ and $S(t_2)^+$ coincide at $p_2$ and $S(t_2)^*$ lies over $S(t_2)^+$ around $p_2$ because both vectors point towards $W$. Since $\kappa>0$,
$$H^*(p_2)=H(q_2)=\kappa z(q_2)+\mu\leq \kappa z(p_2)+\mu=H(p_2).$$
The tangency principle concludes that $\Pi(t_2)$ is a plane of symmetry, which is a contradiction because $\Gamma\subset \s_{\rho}^+$. Therefore, in the method of moving planes, we can follow reflecting with the planes $\Pi(t)$ arriving at $t=0$.

Next, use the reflections about the planes $Q(t)$ including $L$ where $t\in[\pi/2,\pi]$. Assume the existence of the time $t_3$ for which $int(S_Q(t_3)^*)$ leaves to be included in $W$ or $S_Q(t_3)^*$ is not a graph on $\Pi(t_3)$. Then there exists an interior or boundary point $p_3$ where $S_Q(t_3)^*$ and $S_Q(t_3)^+$ are tangent at $p_3$. We claim that $t_3=\pi$. On the contrary, assume that $t_3\in [\pi/2,\pi)$. Let us observe that the reflections $\Psi_t$ increases the height of a point, for $t\in [\pi/2,\pi)$, i.e. if $q\in S_Q(t)^{-}$, then $z(q)\leq z(\Psi_t(q))$. Denote by $H_Q^*$ the mean curvature of $S_Q(t_3)^{*}$. If $p_3=\Psi_{t_3}(q_3)$, then
$$H_Q^{*}(p_3)=H(q_3)=\kappa z(q_3)+\mu\leq\kappa z(\Psi_{t_3}(q_3))+\mu=H(p_3).$$
Again the unit normal vectors of $S_Q(t_3)^*$ and $S_Q(t_3)^{+}$ coincide at $p_3$ and both point towards $W$. Thus $S_Q(t_3)^*\geq S_Q(t_3)^+$ around $p_3$. As in Theorem \ref{t4}, the new case appears when $S_Q(t_3)^*$ touches $S_Q(t_3)^+$ at a point $p_3\in \Gamma$ with
$p_3\in\partial S_Q(t_3)^*\cap \partial S_Q(t_3)^+$. Then both surfaces are tangent at $p_3$ because $S$ meets $\s_\rho$ with constant angle along $\Gamma$. The Serrin corner lemma shows that $Q(t_3)$ is a plane of symmetry of $S$. As $t_3\not=\pi$, $Q(t_3)$ is not a vertical plane. This is a contradiction because the mean curvature $H(x,y,z)=\kappa z+\mu$ changes by reflections about $Q(t_3)$.

This proves that $t_3=\pi$ and $S$ is invariant by reflections about the plane of equation $x=0$. Doing the same argument for all such horizontal straight lines, we deduce that the surface is a surface of revolution with respect to the $z$-axis. Moreover, by the proof, each plane of symmetry separates $S$ in two graphs, which shows that the generating curve is a graph on the $z$-axis. This proves that any nonempty intersection of $S$ with a horizontal plane is a circle with center on the axis. $\Box$
\end{proof}


\emph{Acknowledgements}: Part of this work was realized while the first author was visiting KIAS at Seoul and the Department of Mathematics of the Pusan National University in April of 2012, whose hospitality is gratefully acknowledged.


\end{document}